\documentclass{amsart}
\usepackage{graphicx} 
\usepackage{amssymb,amsmath,latexsym,array,delarray,amsthm,amssymb,epsfig,setspace,multicol,titletoc,tocloft,titlesec,graphicx,mathtools,soul,xcolor,tikz,pgfplots,fancyhdr,tikz}
\usepackage{quiver}

\titleformat{name=\section}{}{\thetitle.}{0.8em}{\centering\scshape}
\titlespacing{\subparagraph}{0em}{0em}{0.5em}

\allowdisplaybreaks

\theoremstyle{definition}

\newtheorem*{theorem*}{Theorem}
\newtheorem*{definition*}{Definition}

\newtheorem{theorem}{Theorem}[section]
\newtheorem{lemma}[theorem]{Lemma}
\newtheorem{prop}[theorem]{Proposition}
\newtheorem{cor}[theorem]{Corollary}
\newtheorem{definition}[theorem]{Definition}
\newtheorem{ex}[theorem]{Example}
\newtheorem{rmk}[theorem]{Remark}

\newcommand{\zz}{\mathbb{Z}}
\newcommand{\nn}{\mathbb{N}}

\newcommand{\cc}{\mathbb{C}}

\newcommand{\ff}{\mathbb{F}}
\newcommand{\bbb}{\mathcal{B}}

\newcommand{\vz}{\vec{z}}
\newcommand{\ph}{\phantom{-}}
\newcommand{\gl}{\mathfrak{gl}}

\newcommand{\ad}{\mathrm{ad}}
\newcommand{\spann}{\mathrm{span}}

\newcommand{\quo}[2]{{\raisebox{.2em}{$#1$}\left/\raisebox{-.2em}{$#2$}\right.}}

\newcommand\scalemath[2]{\scalebox{#1}{\mbox{\ensuremath{\displaystyle #2}}}}

\begin{document}

\title[]{More on characteristic polynomials of Lie algebras}

\author{Korkeat Korkeathikhun}
\address{Korkeat Korkeathikhun\\ Department of Mathematics, Faculty of Science, National University of Singapore, 119076, Singapore} 
\email{\tt korkeatk@nus.edu.sg, korkeat.k@gmail.com}

\author{Borworn Khuhirun}
\address{Borworn Khuhirun\\ Department of Mathematics and Statistics, Faculty of Science and Technology, Thammasat University, Rangsit Campus, Pathum Thani 12121, Thailand} 
\email{\tt borwornk@mathstat.sci.tu.ac.th}

\author{Songpon Sriwongsa}
\address{Songpon Sriwongsa \\ Department of Mathematics \\ Faculty of Science \\ King Mongkut's University of Technology Thonburi (KMUTT) \\ 126 Pracha-Uthit Road \\ Bang Mod, Thung Khru \\ Bangkok 10140, Thailand}
\email {\tt songpon.sri@kmutt.ac.th}

\author{Keng Wiboonton$^*$}
\thanks{*Corresponding Author}
\address{Keng Wiboonton \\ Department of Mathematics and Computer Scienc \\ Faculty of Science \\ Chulalongkorn University \\ Bangkok 10330, Thailand}
\email{\tt keng.w@chula.ac.th}



\begin{abstract}
In recent years, the notion of characteristic polynomial of representations of Lie algebras has been widely studied. This paper provides more properties of these characteristic polynomials. For simple Lie algebras, we characterize the linearization of characteristic polynomials. Additionally, we characterize nilpotent Lie algebras via characteristic polynomials of the adjoint representation.
\end{abstract}

\maketitle

\bigskip

\section{Introduction}


The determinant and eigenvalues of square matrices are classical and important tools in linear algebras and many areas in mathematics. Moreover, both tools are used in
applications in other areas in science and engineering. A basic ingredient of studying eigenvalues and eigenvectors of a square matrix $A$ is its characteristic polynomial defined by $\det (zI - A)$. It is then natural to extend this notion of characteristic polynomial for several square matrices of the same size. One generalization of this notion, which has recently been studied, is defining the characteristic polynomial for square matrices $A_1, \dots, A_n$ of the same size to be $\det(z_0 I + z_1 A_1 + \cdots z_n A_n)$. This generalization is interesting and plays a significant role in many branches of mathematics such as matrix theory, operator theory, group theory and Lie  theory.

A similar genalization of this notion was actually dated back to 1886 when Dedekind computed the group determinant for certain finite non-abelian groups and proved the nice result that, the group
determinant for a finite abelian group factors  as a
product of linear factors over the complex numbers. After communicated by letters from Dedekind, in 1896, Frobenius discovered a way to define characters of general finite groups, developed many theorems about them and used these results to prove that the group determinant for a general finite group can be written as the product of irrecible factors. This study led by Dedekind and Frobenius is undoubtedly the staring point of representation theory of finite groups. For more information of this development, we refer the reader to \cite{CurtisArticle}, \cite{CurtisBook} and the references [12, 13, 17] cited in \cite{spectral}.

In modern day terminology, for a finite group $G = \{1, g_1, \dots, g_n\}$, Dedekind studied the group determinant 
$$\det(z_0I+z_1\lambda_G(g_1)+\cdots+z_n\lambda_G(g_n))$$
where $\lambda_G$ is the left regular representation of $G$. As mention above, this group determinant, which is a homogeneous polynomial in several variables, has a nice factorization. In addition, one may define this polynomial associated to a general finite-dimensional representation of $G$. For further development in this direction, we refer the reader to the references cited in \cite{spectral}. 

By the same token, one may consider this polynomial of a Lie algebra $L$ with respect to a finite-dimensional representation of $L$. In \cite{hu2019determinants}, Hu and Zhang set up the study of characteristic polynomials of Lie algebras with respect to their finite-dimensional representations. For a Lie algebra $L$ with a basis $\{e_1, \dots, e_s\}$ and a finite-dimensional representation $\phi$ of $L$, the corresponding characteristic polynomial is defined in  \cite{hu2019determinants} as
\begin{align*}
p_{L,\phi}(z) :=\det\left(z_0I+z_1\phi(e_1) + \cdots + z_s\phi(e_s)\right).
\end{align*}

With respect to the canonical basis of $\mathfrak{sl}_2(\cc)$, the formula of the characteristic polynomial of $\mathfrak{sl}_2(\cc)$ is obtained in \cite{hu2019determinants}, \cite{chen2019characteristic} and \cite{hu2021eigenvalues}, which is 
\[p_{\phi}(z_0, z_1, z_2, z_3)=
    \begin{cases}
        z_0\displaystyle\prod_{i=1}^{m/2}(z_0^2 - (2i)^2 (z_1^2+z_2 z_3)), &\quad \text{ if $m$ is even,}\vspace{.5 cm}\\ 
        \displaystyle\prod_{i=0}^{(m-1)/2}(z_0^2 -(2i+1)^2 (z_1^2+z_2 z_3)), &\quad \text{ if $m$ is odd,}
    \end{cases}
\]
where $\phi$ is an irreducible representation of $\mathfrak{sl}_2(\cc)$ of highest weight $m$. Later in \cite{jiang2022characteristic}, the authors proved that any finite-dimensional representation of $\mathfrak{sl}_2(\cc)$ is uniquely determined by its characteristic polynomial. In other words, there is a one-to-one correspondence between isomorphism classes of finite-dimensional representations and their characteristic polynomials. Moreover, they gave a monoid structure on the set of all characteristic polynomials which is compatible with the tensor product of representations. The results of one-to-one correspondence and a monoid structure also hold for any complex simple Lie algebras which was proved in \cite{geng2022characteristic}.

In \cite{hu2019determinants}, characteristic polynomials of solvable Lie algebras of finite dimension over an algebraically closed field $\mathbb{F}$ of characteristic zero are studied. One of their main results is that a finite-dimensional Lie algebra over $\ff$ is solvable if and only if the characteristic polynomial is completely reducible with respect to any finite-dimensional representation and any basis. This is a new characterization of solvability of Lie algebras. 

In this work, we first focus on characteristic polynomials of representations of simple Lie algebras and characterize the linearization of such polynomials. Moreover, we study the characterization of 
nilpotent Lie algebras via characteristic polynomials of the adjoint representation.
 Our paper is organized as follows. The definitions and certain properties of characteristic polynomials of a Lie algebra of finite dimension are recalled in Section 2. In Section 3, we assume that a Lie algebra is simple and is over the field of complex numbers $\cc$. We show that the linearization of characteristic polynomials are closely related to characters of representations. Consequently, we give a characterization of the linearization of characteristic polynomials which is independent of prescribed representations. In Section 4, we characterize nilpotent Lie algebras using characteristic polynomials of the adjoint representation. This result is analogous to Theorem 2 in \cite{hu2019determinants}.

\section{Characteristic polynomials}
In this section, we set up our notations, and recall definitions and results concerning characteristic polynomials. Additional details can be found in \cite{geng2022characteristic} or in \cite{hu2019determinants}.
 Let $\ff$ be an algebraically closed field of characteristic 0 and $L$ a finite-dimensional Lie algebra over $\ff$ with a basis $\mathcal{B}=\{e_1,\dots, e_s\}$. Let $\phi:L\to \mathfrak{\mathfrak{gl}}(V)$ be a finite-dimensional representation of $L$. The \textsl{characteristic polynomial of $L$ with respect to $\phi$ and the basis $\bbb$}, denoted by $p_{\phi,\bbb}$ or $p_{V,\bbb}$, is defined by
$$
p_{\phi, \mathcal{B}}(\vz)=p_{\phi, \mathcal{B}}(z_0, z_1, \dots, z_s):=\det\left(z_0 I+\sum_{i=1}^s z_i\phi(e_i)\right)
$$
where $\vz=(z_0, z_1, \dots, z_s)$.\\
\indent In particular, if $L$ is a Lie subalgebra of $\mathfrak{gl}(V)$ with a basis $\mathcal{B}=\{e_1,\dots, e_s\}$, then \textsl{the characteristic polynomial of $L$ with respect to the basis $\bbb$}, denoted by $p_{L,\bbb}$, is defined by
\begin{align*}
p_{L,\bbb}(\vz)=p_{L,\bbb}(z_0,z_1,\ldots,z_s):=\det\left(z_0I+\sum_{i=1}^s z_ie_i\right).
\end{align*}
\indent It follows immediately from the properties of the determinant that characteristic polynomial is a homogeneous polynomial of degree equal to $\dim V$ and invariant under change of basis of $V$. The characteristic polynomial is said to be {\it completely reducible} if it can be factored into a product of linear polynomials.

Let $\mathcal{B}_1, \mathcal{B}_2$ be bases of $L$ and $P$ a transition matrix from $\mathcal{B}_2$ to $\mathcal{B}_1$. By \cite[Theorem 2.2]{geng2022characteristic}, we have
$$p_{\phi, \mathcal{B}_2}(\vz)=p_{\phi, \mathcal{B}_1}(\vz D)$$
where
\[D=
\begin{pmatrix}
    1 & 0 \\
    0 & P^\top
\end{pmatrix}.
\]
Hence $p_{\phi, \mathcal{B}_1}(\vz)$ and $p_{\phi, \mathcal{B}_2}(\vz)$ coincide up to linear change of variables. In particular, they have the same reducibility. Then we can fix a basis of $L$ and simply write $p_\phi(\vz)$ or $p_L(\vz)$. 
\begin{ex}
    Let $M_1^{10}=\spann\;\bbb_1=\spann\{e_1,e_2,e_3,e_4\}$ be a Lie algebra defined by 
    $$
    [e_4,e_1]=e_2, ~[e_4,e_2]=e_1, ~[e_3,e_1]=e_1 \text{ and } [e_3,e_2]=e_2
    $$
    (See \cite{solv}). Then we consider the adjoint representation $\ad:L\to\gl(L)$ so that
    \begin{align*}
    \ad_{e_1}=-E_{13}-E_{24}, ~\ad_{e_2}=-E_{14}-E_{23}, ~\ad_{e_3}=E_{11}+E_{22}, ~\ad_{e_4}=E_{12}+E_{21},
    \end{align*}
    where $E_{ij}$ denote the $n\times n$ elementary matrix whose $(i,j)$-entry is 1 and all other entries are zero.
    Thus the characteristic polynomial
    $$
    p_{\ad,\bbb_1}(z)=
    \det
        \begin{pmatrix}
            z_0+z_3&z_4&-z_1&-z_2\\
            z_4&z_0+z_3&-z_2&-z_1\\
            0&0&\ph z_0&\ph0\\
            0&0&\ph0&\ph z_0
        \end{pmatrix}
    =z_0^2(z_0+z_3+z_4)(z_0+z_3-z_4).
    $$
    On the other hand, let $x_1=e_1+e_2, x_2=e_2+e_3, x_3=e_3+e_4$ and $x_4=e_4$. As $M_1^{10}=\spann\;\bbb_2=\spann\{x_1,x_2,x_3,x_4\}$, we have
    \begin{align*}
        \ad_{x_1}&=\ad_{e_1}+\ad_{e_2}=-E_{13}-E_{24}-E_{14}-E_{23},\\
        \ad_{x_2}&=\ad_{e_2}+\ad_{e_3}=-E_{14}-E_{23}+E_{11}+E_{22},\\
        \ad_{x_3}&=\ad_{e_3}+\ad_{e_4}=E_{11}+E_{22}+E_{12}+E_{21},\\
        \ad_{x_4}&=\ad_{e_4}=E_{12}+E_{21},
    \end{align*}
    Hence the characteristic polynomial
    $$
    p_{\ad,\bbb_2}(\vz)=\det
        \begin{pmatrix}
            z_0+z_2+z_3&z_3+z_4&-z_1&-z_1-z_2\\
            z_3+z_4&z_0+z_2+z_3&-z_1-z_2&-z_1\\
            0&0&\ph z_0&\ph0\\
            0&0&\ph0&\ph z_0
        \end{pmatrix}
    $$
    is equal to $z_0^2(z_0+z_2+2z_3+z_4)(z_0+z_2-z_4)$, which has the same reducibility as $p_{\ad,\bbb_1}(\vz)$. Observe that the transition matrix from $\bbb_2$ to $\bbb_1$ is
    $$
    P=
    \begin{pmatrix}
    1&0&0&0\\
    1&1&0&0\\
    0&1&1&0\\
    0&0&1&1
    \end{pmatrix}
    ~\text{ and }~
    D=
    \begin{pmatrix}
    1 & 0 \\
    0 & P^\top
    \end{pmatrix}
    =
    \begin{pmatrix}
    1&0&0&0&0\\
    0&1&1&0&0\\
    0&0&1&1&0\\
    0&0&0&1&1\\
    0&0&0&0&1
    \end{pmatrix},
    $$
    so
    \begin{align*}
        p_{\ad,\bbb_1}(\vz D)=p_{\ad,\bbb_1}(z_0,z_1,z_1+z_2,z_2+z_3,z_3+z_4)=p_{\ad,\bbb_2}(\vz).
    \end{align*}
\end{ex}

The set of all characteristic polynomial of $L$ with respect to a representation $\phi$ is denoted by
$$\mathbf{CP}_L=\{p_\phi\mid \phi \text{ is a representation of } L\}.$$
\indent For any representations $U$ and $V$ of $L$, it follows from the definition that 
\begin{align}\label{. in CP}
    p_{U \oplus V}=p_{U}\cdot p_{V}.
\end{align}
Moreover, the characteristic polynomial of a representation $\phi$ and its dual $\phi^*$ are related as described in the following proposition. 


\begin{prop}
    Let $p_{\phi}\in \mathbf{CP}_L$ for some representation $\phi$ of $L$ of dimension $k$. Then the characteristic polynomial of the dual representation $p_{\phi^*}$ of $\phi$ is
    $$p_{\phi^*}(z_0, z_1, \dots, z_k)=(-1)^{\dim \phi}p_{\phi}(-z_0, z_1, \dots, z_k).$$
\end{prop}
\begin{proof}
    It is known that $\phi^*(x)=-\phi(x)^\top$ for any $x\in L$ where $\top$ is the transpose. Let $\{x_1,\dots, x_k\}$ be a basis of $L$. Hence
    \begin{align*}
        p_{\phi^*}&=\det\left(z_0 I + \sum_{i=1}^{k}z_i\phi^*(x_i)\right)\\
        &=\det\left(z_0 I - \sum_{i=1}^{k}z_i\phi(x_i)^\top\right)\\
        &=(-1)^{\dim \phi}\det\left(-z_0 I + \sum_{i=1}^{k}z_i\phi(x_i)\right)^\top\\
        &=(-1)^{\dim \phi}p_{\phi}(-z_0, z_1, \dots, z_k).
    \end{align*}
    as desired.
\end{proof}

\section{The linearized characteristic polynomials and the Weyl groups}
In this section, we fulfill the detail from \cite{geng2022characteristic} about characteristic polynomials of a finite-dimensional complex simple Lie algebra $L$ of rank $\ell$. We first recall certain definitions and facts about representation rings and characters which can be found in \cite{fulton2013representation} or \cite{humphreys2012introduction}. At the end of this section, we show that the image of $\rho$ defined in \cite{geng2022characteristic} can be characterized via the action of the Weyl group.

Let $\Phi, \Lambda$ and $W$ be the root system, the weight lattice and the Weyl group of $L$, respectively. Fix a cartan subalgebra $H$ with a fixed basis $\{h_1,\dots, h_\ell\}$. We fix a canonical basis $\{h_1, \dots, h_\ell, E_\alpha : \alpha\in \Phi\}$ of $L$ so that
$$p_\phi=\det\left(z_0 I + \sum_{i=1}^\ell z_i \phi(h_i) + \sum_{\alpha\in \Phi}z_{\alpha}\phi(E_\alpha)\right)$$
for any representation $\phi$ of $L$. 

The representation ring $R(L)$ of $L$ is a ring whose elements are isomorphism classes $[V]$ of finite-dimensional representations of $L$ and the addition and multiplication are given by
\begin{align*}
    [U]+[V]&=[U\oplus V]\\
    [U]\cdot [V]&= [U\otimes V]
\end{align*}
for any representations $U$ and $V$ of $L$. By \cite[Theorem 3.3]{geng2022characteristic}, the map 
\begin{align*}
    p:R(L)&\to \mathbf{CP}_L\\
    [V]&\mapsto p_V
\end{align*}
is well-defined and bijective.

Let $\mathbb{Z}[\Lambda]$ be the integral group ring on the abelian group $\Lambda$. The elements of $\mathbb{Z}[\Lambda]$ are formal sums
$\sum_{\lambda\in \Lambda} n_{\lambda}e^{\lambda}$ where $n_\lambda\in\mathbb{Z}$ and all $n_\lambda$ but finitely many are zeros. The addition is canonically defined and the multiplication is given by $e^\lambda e^\mu=e^{\lambda+\mu}$ for any $\lambda, \mu\in\Lambda$. A \textsl{character homomorphism} $\mathrm{ch}:R(L)\to \mathbb{Z}[\Lambda]$ is defined by
$$\mathrm{ch}([V])=\sum_{\lambda}(\dim V_\lambda)e^\lambda$$
where 
$$V_\lambda=\{v\in V\mid h\cdot v=\lambda(h)v \text{ for all } h\in H\}.$$ 
If $V_\lambda\neq 0$, we call $V_\lambda$ a weight space and call $\lambda$ a weight of $V$. Let $\Pi(V)$ be the set of all weights of $V$. It is known that the Weyl group $W$ permutes $\Pi(V)$ and $\dim V_\mu=\dim V_{\sigma \mu}$ for all $\sigma \in W$.
Note that $W$ naturally acts on $\zz[\Lambda]$. Let $\zz[\Lambda]^W$ be the $W$-invariant subring. By \cite[Theorem 23.24]{fulton2013representation}, the map
$\mathrm{ch}:R(L)\to \zz[\Lambda]^W$
is a ring isomorphism.

In \cite{geng2022characteristic}, for each representation $\phi$ of $L$, the \textsl{linearization} $\tilde{p}_\phi$ of $p_\phi$ is defined by
$$\tilde{p}_\phi(z_0, z_1, \dots, z_\ell):=p_\phi (z_0, z_1, \dots, z_\ell, 0,0, \dots 0).$$
The set of all linearizations is denoted by
$$\widetilde{\mathbf{CP}}_L=\{\tilde{p}_\phi\mid \phi\text{ is a representation of } L\}.$$

\begin{prop}(\cite[Proposition 3.2]{geng2022characteristic})\label{rho is bijective}
    The map 
\begin{align*}
    \rho:\mathbf{CP}_L &\to \widetilde{\mathbf{CP}}_L\\
    \rho(p_\phi)&=\tilde{p}_\phi
\end{align*}
is bijective.
\end{prop}

\begin{ex}
    Consider $L=\mathfrak{sl}_2$. The explicit map $\rho$ is given by
\begin{align*}
\rho(p_\phi)=
    \begin{cases}
        z_0\displaystyle\prod_{i=1}^{m/2}(z_0^2 - (2i)^2 z_1^2), &\quad \text{ if $m$ is even,}\vspace{.5 cm}\\ 
        \displaystyle\prod_{i=0}^{(m-1)/2}(z_0^2 -(2i+1)^2 z_1^2), &\quad \text{ if $m$ is odd,}
    \end{cases}
\end{align*}
    where $\phi$ is an irreducible representation of $L$ of highest weight $m$.
\end{ex}

More general, since $h_i\cdot v=\lambda(h_i)v$ for each weight vector $v\in V_\lambda$, $\phi(h_i)$ is a diagonal matrix with respect to a weight basis of $V$. Hence
\begin{align*}
    \tilde{p}_\phi(z_0, z_1, \dots, z_\ell)&=\det\left( z_0 I + \sum_{i=1}^\ell z_i\phi(h_i)\right)
    =\prod_{\lambda\in \Pi(\phi)}\left(z_0 + \sum_{i=1}^\ell \lambda(h_i)z_i \right)^{\dim V_\lambda}.
\end{align*}
Since the linearization preserves the multiplication, we have
$$\tilde{p}_U\cdot \tilde{p}_V=\tilde{p}_{U\oplus V}$$
for any representations $U$ and $V$ of $L$.

\begin{definition}
    Let $U$ and $V$ be representations of $L$. The \textsl{resolution product} $\tilde{p}_U * \tilde{p}_V$ of $\tilde{p}_U$ and $\tilde{p}_V$ is the polynomial defined by
    $$\tilde{p}_U* \tilde{p}_V=\prod_{\lambda\in \Pi(U), \mu\in \Pi(V)}\left(z_0 + \sum_{i=1}^\ell (\lambda+\mu)(h_i)z_i \right)^{\dim V_\lambda\cdot\:\dim V_\mu}.$$
\end{definition}

\begin{prop}(\cite[Proposition 4.2]{geng2022characteristic})\label{* is closed}
    Let $U$ and $V$ be any representations of $L$. Then
    $$\tilde{p}_U * \tilde{p}_V=\tilde{p}_{U\otimes V}.$$
\end{prop}

Hence, one can define the resolution product $p_U * p_V$ of $p_U, p_V\in \mathbf{CP}_L$ as the polynomial
$$p_U* p_V=\rho^{-1}(\tilde{p}_U* \tilde{p}_V)$$
which is well-defined as $\rho$ is bijective. It follows that
\begin{align}\label{* in CP}
    p_U *p_V=p_{U\otimes V}.
\end{align}

\begin{prop}
    $(\widetilde{\mathbf{CP}}_L,\cdot, *)$ and $(\mathbf{CP}_L,\cdot, *)$ are commutative rings with the identity $z_0$. Moreover, $\rho$ is a ring isomorphism.
\end{prop}
\begin{proof}
It is obvious that $\widetilde{\mathbf{CP}}_L$ is an abelian group under the usual multiplication $\cdot$ and the identity is $1$ which corresponds to the zero space. Note that the trivial representation has the characteristic polynomial $z_0$. It is clear from the definition of $*$ that $\tilde{p}_U*z_0=\tilde{p}_U$ for any representation $U$. By Proposition \ref{* is closed}, $\widetilde{\mathbf{CP}}_L$ is a commutative monoid under $*$ with the identity $z_0$. Thus, only the distributive law of $*$ over $\cdot$ is left to be shown. Let $p_U, p_V, p_W \in \widetilde{\mathbf{CP}}_L$. Then
\begin{align*}
    (\tilde{p}_U\cdot \tilde{p}_V)*\tilde{p}_W &= \tilde{p}_{U\oplus V}* \tilde{p}_W\\
    &=\tilde{p}_{(U\oplus V) \otimes W}\\
    &=\tilde{p}_{(U\otimes W)\oplus (V\otimes W)}\\
    &=\tilde{p}_{U\otimes W}\cdot \tilde{p}_{V\otimes W}\\
    &=(\tilde{p}_{U}*\tilde{p}_{W})\cdot(\tilde{p}_{V}*\tilde{p}_{W}),
\end{align*}
as desired. The argument is the same for $\mathbf{CP}_L$. By Proposition \ref{rho is bijective}, $\rho$ is bijective. It remains to show that $\rho$ is a ring homomorphism. In fact,
    $$\rho(p_U\cdot p_V)=\rho(p_{U\oplus V})=\tilde{p}_{U\oplus V}=\tilde{p}_U\cdot \tilde{p}_V=\rho(p_U)\cdot \rho(p_V)$$
    and, by the definition of $*$ on $\mathbf{CP}_L$,
    $$\rho(p_U*p_V)=\rho(\rho^{-1}(\tilde{p}_U* \tilde{p}_V))=\tilde{p}_U* \tilde{p}_V=\rho(p_U)*\rho(p_V).$$
    This completes the proof.
\end{proof}

\begin{cor}
    The map $p:R(L)\to \mathbf{CP}_L$ is a ring isomorphism.
\end{cor}
\begin{proof}
    This follows from (\ref{. in CP}), (\ref{* in CP}) and the fact that $p$ is bijective.
\end{proof}

Let $X(\Lambda)$ be the set of polynomials of the form
$$\prod_{i \in I}(z_0+a_{i,1} z_1+\dots+ a_{i,n} z_n)^{i}$$
where $I\subset \nn\cup\{0\}$ is a finite index set and $a_{i,1},\dots, a_{i,n}\in \zz$. Since each weight $\lambda\in\Lambda$ is uniquely determined by its integral values $\lambda(h_i), i=1,\dots, n$, we have that $\Lambda$ and $\zz^n$ are one-to-one correspondence. Hence elements of $X(\Lambda)$ can be written as a finite product
\begin{align}\label{element in X^W}
    \prod_{\lambda\in \Lambda}(z_0+\lambda(h_1) z_1+\dots+ \lambda(h_n) z_n)^{i_\lambda}
\end{align}
where all $i_\lambda$ but finitely many are zeros. This justifies the notation $X(\Lambda)$. Define the action of $W$ on $X$ by
$$w\cdot \prod_{\lambda\in \Lambda}(z_0+\lambda(h_1) z_1+\dots+ \lambda(h_n) z_n)^{i_\lambda}= \prod_{\lambda\in \Lambda}(z_0+w\lambda(h_1) z_1+\dots+ w\lambda(h_n) z_n)^{i_{w\lambda}}$$
for any $w\in W$. The following theorem is one of the main results of this paper.


\medskip

\begin{theorem}\label{main thm 1}
    We have
    $$\widetilde{\mathbf{CP}}_L=X(\Lambda)^W.$$
\end{theorem}
\begin{proof}
    For any $\tilde{p}_\phi=\prod_{\lambda\in \Pi(\phi)} \left(z_0+\sum_{i=1}^n \lambda(h_i)z_i\right)^{\dim V_\lambda}\in\widetilde{\mathrm{CP}}_L$, since $W$ acts on $\Pi(\phi)$ and $\dim V_\lambda=\dim V_{w\lambda}$ for all $w\in W$, $\tilde{p}_\phi\in X^W$. To prove the converse, let $f\in X^W$ be a polynomial of the form (\ref{element in X^W}). Then the formal sum $\sum_{\lambda\in \Lambda}i_\lambda e^\lambda$ is also an element of $\zz[\Lambda]^W$. Let $\phi':L\to \gl(V)$ be the unique (up to isomorphism) representation of $L$ corresponding to this formal sum under the character map $\mathrm{ch}$. Then nonzero $i_\lambda$'s are exactly the weights of $\phi'$ and $\dim V_\lambda=i_\lambda$. By definition, $\tilde{p}_\lambda=f$ and so $f\in \widetilde{\mathbf{CP}}_L$.
\end{proof}

\begin{rmk}
Let $\varphi=\rho\circ p \circ \mathrm{ch}^{-1}:\zz[\Lambda]^W\to \widetilde{\mathbf{CP}_L}$. By Theorem \ref{main thm 1}, $\varphi$ can be expressed explicitly as
$$\varphi\left(\sum_{\lambda}n_\lambda e^\lambda\right)=\prod_\lambda \left(z_0+\sum_{i=1}^\ell \lambda(h_i)z_i\right)^{n_\lambda}.$$
In particular, $\varphi$ is a ring isomorphism. It immediately follows from the definition of  $\varphi$ that the diagram
\[
    \begin{tikzcd}
	{R(L)} &&& {\mathbf{CP}_L} \\
	\\
	{\mathbb{Z}[\Lambda]^W} &&& {X(\Lambda)^W}
	\arrow["{\mathrm{ch}}"', from=1-1, to=3-1]
	\arrow["p", from=1-1, to=1-4]
	\arrow["\rho", from=1-4, to=3-4]
	\arrow["\varphi", from=3-1, to=3-4]
    \end{tikzcd}
    \]
commutes.
\end{rmk}


    
\begin{ex}
    Consider $L=\mathfrak{sl}_n$. Fix a cartan subalgebra $H$ with a fixed basis $\{h_i=E_{i,i}-E_{i+1,i+1}: i=1,\dots, n-1\}$. For each $i=1,\dots, n$, let $\varepsilon_i: \mathfrak{sl}_n \to \mathbb{C}$ be given by $\varepsilon_i(A)=a_{ii}$ for $A=(a_{ij})\in \mathfrak{sl}_n$. For $i=1,\dots , n-1$, let
    $$\alpha_i=\varepsilon_i - \varepsilon_{i+1},$$
    $$\omega_i=\varepsilon_1+\varepsilon_2+\cdots +\varepsilon_i$$
    be simple roots and a fundamental weights of $\mathfrak{sl}_n$, respectively. Let $\Lambda^+$ be the set of dominant integral weights of $\mathfrak{sl}_n$ and $V(\lambda)$ the irreducible representation corresponding to $\lambda\in \Lambda^+$. One can associate each $$\lambda=a_1\omega_1+ a_2\omega_2+ \cdots + a_{n-1}\omega_{n-1}\in \Lambda^+ (a_i\in \mathbb{Z}_{\geq 0})$$ the partition $(\lambda_1 \geq  \lambda_2\geq \dots\geq \lambda_{n-1} \geq 0)$ where
    \begin{align*}
       \lambda_1 &= a_1+ a_2+ \cdots + a_{n-1}\\
       \lambda_2 &= a_2+ \cdots + a_{n-1}\\ 
       &\  \vdots\\
      \lambda_{n-1} &= a_{n-1}.
    \end{align*}
    In other words, $\lambda_i$ equals to the coefficient of $\varepsilon_i$. The character of $V(\lambda)$ is the schur polynomial $s_\lambda(x_1,\dots, x_{n-1})$ where $x_i=e^{\varepsilon_i}$. The weights of $V(\lambda)$ are 
    $$a_1\varepsilon_1 + a_2\varepsilon_2 + \cdots + a_{n-1}\varepsilon_{n-1}$$
    whose multiplicities is the coefficient of the monomial $x_1^{a_1}x_2^{a_2}\cdots x_{n-1}^{a_{n-1}}$ occurring in $s_\lambda(x_1,\dots, x_{n-1})$. In particular, the character of the fundamental representation $V(\omega_k)$ is the $k$-th elementary symmetric function
    $$\sum_{i_1<\cdots<i_k} x_{i_1}\cdots x_{i_k}=\sum_{i_1<\cdots<i_k} e^{\varepsilon_{i_1}+\cdots + \varepsilon_{i_k}}.$$
    Here, the multiplicity of each weight $\varepsilon_{i_1}+\cdots + \varepsilon_{i_k}$ is 1. Hence
    \begin{align*}
        \tilde{p}_{V(\omega_k)}&=\varphi\left(\sum_{i_1<\cdots<i_k} e^{\varepsilon_{i_1}+\cdots + \varepsilon_{i_k}}\right)\\
        &=\prod_{i_1<\cdots<i_k} \left(z_0+\sum_{i=1}^{n-1} (\varepsilon_{i_1}+\cdots + \varepsilon_{i_k})(h_i)z_i\right).\\
    \end{align*}
\end{ex}


\section{Nilpotent Lie algebras}

Assume that $V$ is an $n$-dimensional vector space over $\ff$. The following theorems are characterization of solvable Lie algebras via characteristic polynomials given by Hu and Zhang \cite{hu2019determinants}.

\begin{theorem}(\cite[Theorem 5.1]{hu2019determinants})\label{reducible1}
    Let $L$ be a subalgebra of $\mathfrak{gl}(V)$ with basis $\{e_1,e_2,\ldots,e_s\}$. Then $L$ is solvable if and only if the characteristic polynomial $p_L(\vz) := \det\left(z_0I+\sum_{i=1}^s z_ie_i\right)$ is completely reducible.
\end{theorem}

\begin{theorem}(\cite[Theorem 5.4]{hu2019determinants})\label{reducible2}
    Let $L$ be a finite-dimensional Lie algebra. Then $L$ is solvable if and only if the characteristic polynomial of $L$ is completely reducible with respect to any finite-dimensional representation and any basis.
\end{theorem}

 Next, we consider the adjoint representation of a Lie algebra. Recall that if $\{e_1,e_2,\ldots,e_s\}$ is a basis of a Lie algebra $L$, then the characteristic polynomial of $L$ with respect to the adjoint representation is $p_\ad(\vz):=\det(z_0I+\sum_{i=1}^s z_i \ad_{e_i})$.

\begin{lemma}\label{z0}
    Let $L$ be a Lie algebra with a basis $\{e_1,e_2,\ldots,e_s\}$. If $[L,L]\subsetneq L$, then $z_0^k$ is a factor of $p_\ad(\vz)$ where $k$ is the codimension of $[L,L]$ in $L$.
\end{lemma}

\begin{proof}
Suppose that $L':=[L,L]\subsetneq L$. If $L'=\{0\}$, then $L$ is abelian, so $\ad_x=0$ for all $x\in L$. Thus $p_\ad(\vz)=\det(z_0I+\sum_{i=1}^s z_i \ad_{e_i})=z_0^n$. Assume that $L'\neq\{0\}$. Let $\{x_1,x_2,\ldots,x_t\}$ be a basis of $L'$ where $k:=s-t\leq1$ and extend this basis to a basis $\{x_1,x_2,\ldots,x_t,y_{t+1},\ldots,y_s\}$ of $L$. Since $\ad_x(L)\subseteq[L,L]=L'$ for every $x\in L$, the adjoint matrix $\ad_x$ is of the form
$$
\begin{pmatrix}
A_{11}&A_{12}\\
0&0
\end{pmatrix}
$$
where $A_{11}\in \text{M}_{t\times t}(\ff)$ and $A_{12}\in \text{M}_{t\times (s-t)}(\ff)$. Thus 
$$
z_0I+\sum_{i=1}^s z_i \ad_{e_i}=
\begin{pmatrix}
A_{11}'(\vz) & A_{12}'(\vz)\\
0&z_0I
\end{pmatrix}
$$
where $A_{11}'(\vz)\in \text{M}_{t\times t}(\ff[\vz])$ and $A_{12}'(\vz)\in \text{M}_{t\times k}(\ff[\vz])$. Consequently,
$$
p_\ad(\vz)=\det(z_0I+\sum_{i=1}^s z_i \ad_{e_i})=z_0^k\det(A_{11}'(\vz))
$$
as desired.
\end{proof}

If $L$ is solvable, then $[L,L]$ is a proper subalgebra of $L$. Then the following corollary immediately follows.

\begin{cor}\label{solvabledet}
    Let $L$ be a Lie algebra with a basis $\{e_1,e_2,\ldots,e_s\}$. If $L$ is solvable, then $z_0^k$ is a factor of the characteristic polynomial $p_\ad(\vz)$ where $k$ is the codimension of $[L,L]$ in $L$.
\end{cor}

The converse of Corollary \ref{solvabledet} does not hold in general. Therefore this result is not enough to characterize solvable Lie algebras, which will be illustrated in the following example.

\begin{ex}
Let $L_5=\spann\{e_1,e_2,e_3,e_4,e_5\}$ be a non-solvable Lie algebra defined by
$[e_1,e_2]=e_1, [e_1,e_3]=2e_2, [e_2,e_3]=e_3$ and $[e_4,e_5]=e_4$ (see \cite{orbit}). Then we have
\begin{multicols}{2}
    \noindent
    \begin{align*}
    \scalemath{0.85}{
        \ad_{e_1}=
        \begin{pmatrix}
            \ph0&\ph1&\ph0&\ph0&\ph0\ph\\
            \ph0&\ph0&\ph2&\ph0&\ph0\ph\\
            \ph0&\ph0&\ph0&\ph0&\ph0\ph\\
            \ph0&\ph0&\ph0&\ph0&\ph0\ph\\
            \ph0&\ph0&\ph0&\ph0&\ph0\ph
        \end{pmatrix},
    }\\
    \scalemath{0.85}{
        \ad_{e_2}=
        \begin{pmatrix}
            -1&\ph0&\ph0&\ph0&\ph0\ph\\
            \ph0&\ph0&\ph0&\ph0&\ph0\ph\\
            \ph0&\ph0&\ph1&\ph0&\ph0\ph\\
            \ph0&\ph0&\ph0&\ph0&\ph0\ph\\
            \ph0&\ph0&\ph0&\ph0&\ph0\ph
        \end{pmatrix},
    }\\
    \scalemath{0.85}{
        \ad_{e_3}=
        \begin{pmatrix}
            \ph0&\ph0&\ph0&\ph0&\ph0\ph\\
            -2&\ph0&\ph0&\ph0&\ph0\ph\\
            \ph0&-1&\ph0&\ph0&\ph0\ph\\
            \ph0&\ph0&\ph0&\ph0&\ph0\ph\\
            \ph0&\ph0&\ph0&\ph0&\ph0\ph
        \end{pmatrix},
    }\\
    \scalemath{0.85}{
        \ad_{e_4}=
        \begin{pmatrix}
            \ph0&\ph0&\ph0&\ph0&\ph0\ph\\
            \ph0&\ph0&\ph0&\ph0&\ph0\ph\\
            \ph0&\ph0&\ph0&\ph0&\ph0\ph\\
            \ph0&\ph0&\ph0&\ph0&\ph1\ph\\
            \ph0&\ph0&\ph0&\ph0&\ph0\ph\\
        \end{pmatrix},
    }\\
    \scalemath{0.85}{
        \ad_{e_5}=
        \begin{pmatrix}
            \ph0&\ph0&\ph0&\ph0&\ph0\ph\\
            \ph0&\ph0&\ph0&\ph0&\ph0\ph\\
            \ph0&\ph0&\ph0&\ph0&\ph0\ph\\
            \ph0&\ph0&\ph0&-1&\ph0\ph\\
            \ph0&\ph0&\ph0&\ph0&\ph0\ph\\
        \end{pmatrix}
    \phantom{,}
    }
    \end{align*}
\end{multicols}
and the characteristic polynomial
\begin{align*}
    p_\ad(\vz)=\det
        \begin{pmatrix}
            z_0-z_2&\ph z_1&0&0&0\ph\\
            -2z_3&\ph z_0&2z_1&0&0\ph\\
            0&-z_3&z_0+z_2&0&0\ph\\
            0&\ph0&0&z_0-z_5&z_4\ph\\
            0&\ph0&0&0&z_0\ph
        \end{pmatrix}
\end{align*}
is equal to $z_0(z_0^4-z_0^3z_5-z_0^2z_2^2+3z_0^2z_1z_3+z_0z_2^2z_5+z_0z_1z_2z_3-3z_0z_1z_3z_5-z_1z_2z_3z_5)$.
Observe that the codimension of $[L,L]$ in $L$ is 1, while $z_0$ is also a factor of $p_\ad(\vz)$.
\end{ex}


For nilpotent Lie algebras, we have a partial result analogous to Theorem \ref{reducible1}.

\begin{prop}\label{nilpotent1}
    Let $L$ be a subalgebra of $\gl(V)$ with basis $\{e_1,e_2,\ldots,e_s\}$. If the characteristic polynomial $p_L(\vz) := \det\left(z_0I+\sum_{i=1}^s z_ie_i\right)$ can be written as $(f(\vz))^n$ where $f(\vz)\in\ff[\vz]$ is a homogeneous polynomial of degree 1, then $L$ is nilpotent.
\end{prop}

\begin{proof}
Assume that the characteristic polynomial $p_L(\vz) = \det(z_0I+\sum_{i=1}^s z_ie_i)$ can be written as $(f(\vz))^n$ where $f(\vz)\in\ff[\vz]$ is a homogeneous polynomial of degree 1. Since $p_L(\vz)$ is completely reducible, Theorem \ref{reducible1} implies that $L$ is solvable. By Lie's theorem, there exists a basis $\bbb$ of $V$ such that the matrices of $L$ can be simultaneously upper triangularized. For each $i \in \{1, \dots, s\}$, let $(d_{i1}, \dots, d_{in})$ be the diagonal entry of the matrix of $e_i$. Then
$$(f(\vz))^n \ = \ p_L(\vz) \ = \ \left(z_0+\sum_{i=1}^s d_{i1}z_i\right) \cdots \left(z_0+\sum_{i=1}^s d_{in}z_i\right).$$
Since $\ff[\vz]$ is a UFD, $\sum_{i=1}^s d_{ij}z_i = \sum_{i=1}^s d_{ik}z_i$ for all $j,k \in \{1, \dots, n\}$. Therefore, for each $i \in \{1, \dots, s\}$, $d_{ij} = d_{ik}$ for all $j,k \in \{1, \dots, n\}$. This implies that for each $i \in \{1, \dots, s\}$, the matrix of $e_i$ is of the form $a_iI+A_i$ where $a_i\in\ff$ and $A_i$ is a strictly upper triangular matrix. Now, we see that for all scalars $\alpha_1, \dots, \alpha_s, \beta_1, \dots, \beta_s$ in $\ff$, the commutators
\begin{eqnarray*}
\left[\sum_{i=1}^s \alpha_ie_i, \sum_{j=1}^s \beta_je_j \right] &=& \left[\sum_{i=1}^s \alpha_i(a_iI+A_i), \ \sum_{j=1}^s \beta_j(a_jI+A_j) \right] \\
                &=& \sum_{i,j=1}^s \ (\alpha_i\beta_j)[A_i,A_j]
\end{eqnarray*}
are strictly upper triangular matrices. Thus, all linear combinations of these commutators are also strictly upper triangular matrices. So, with respect to the basis $\bbb$, the set of all matrices of $[L,L]$ is a Lie subalgebra of $\mathfrak{n}(n, \ff)$. Again, for any strictly upper triangular matrix $A\in[L,L]$ and $\alpha_1, \dots, \alpha_s\in\ff$, we have 
$$
\left[\sum_{i=1}^s \alpha_ie_i, A \right] = \left[\sum_{i=1}^s \alpha_i(a_iI+A_i), A \right] = \sum_{i=1}^s \alpha_i[A_i,A] = \left[\sum_{i=1}^s \alpha_iA_i, A \right]
$$
as $I\in Z(\gl(V))$. Let $L_0$ be a Lie algebra spanned by $\{A_1,A_2,\ldots,A_s\}$. Then $L^j=(L_0)^j$ for any $j\geq2$. Clearly, $L_0\subseteq\mathfrak{n}(n, \ff)$ is nilpotent.  Hence, $L$ is also nilpotent. This completes the proof.
\end{proof}



The converse of Proposition \ref{nilpotent1} does not hold in general. For example, we consider $L=\spann\{E_{11}\}$ and $\dim V=n>1$. Clearly, $L$ is nilpotent while the characteristic polynomial is $(z_0+z_1)z_0^{n-1}$.

\begin{lemma}\label{D}
    Let $L$ be a Lie algebra with basis $\{e_1,e_2,\ldots,e_s\}$ and $\phi$ a representation of $L$. Then 
    $$
    p_{\phi}(\vz)=p_{\phi(L)}(\vz D)
    $$
    for some invertible matrix $D$ of size $n+1$.
\end{lemma}
\begin{proof}
Let $\{x_1,x_2,\ldots,x_t\}$ be a basis of $\ker\phi$. Then we extend this basis to a basis $\bbb_1=\{y_1,\ldots,y_{s-t},x_1,x_2,\ldots,x_t\}$ of $L$. Let $P$ be a transition matrix from $\bbb_2=\{e_1,e_2,\ldots,e_s\}$ to $\bbb_1=\{y_1,\ldots,y_{s-t},x_1,x_2,\ldots,x_t\}$. Therefore 
$$
p_{\phi}(\vz):=p_{\phi,\bbb_2}(\vz)=p_{\phi,\bbb_1}(\vz D)
$$
where 
$$
D=
\begin{pmatrix}
    1 & 0 \\
    0 & P^\top
\end{pmatrix}.
$$
On the other hand, since $\{\phi(y_1),\ldots,\phi(y_{s-t})\}$ is a basis of $\phi(L)$ and $\phi(x_i)=0$ for all $i=1,2,\ldots,t$, we have
\begin{align*}
    p_{\phi(L)}(\vz)&=\det\left(z_0I+\sum_{i=1}^{s-t} z_i\phi(y_i)\right)\\
    &=\det\left(z_0I+\sum_{i=1}^{s-t} z_i\phi(y_i)+\sum_{j=1}^t z_{s-t+j}\phi(x_j)\right)\\
    &=p_{\phi,\bbb_1}(\vz).
\end{align*}
Hence $p_{\phi}(\vz):=p_{\phi,\bbb_2}(\vz)=p_{\phi,\bbb_1}(\vz D)=p_{\phi(L)}(\vz D)$. This completes the proof.
\end{proof}

\medskip

Now we have the following characterization of a nilpotent Lie algebra via its characteristic polynomial associated to the adjoint representation. This is another main result of this paper.

\medskip

\begin{theorem}\label{nilpotent}
    Let $L$ be a Lie algebra with a basis $\{e_1,e_2,\ldots,e_s\}$. Then $L$ is nilpotent if and only if $p_\ad(\vz)=z_0^n$.
\end{theorem}

\begin{proof}
Suppose that $L$ is nilpotent. By Engel's theorem, $\ad_x$ is nilpotent for every $x\in L$. Thus $\ad(L)$ is a Lie subalgebra of $\gl(L)$ of nilpotent transformations. By Engel's theorem, there exists a basis $\bbb$ of a vector space $L$ such that the matrices of $L$ can be simultaneously strictly upper triangularized. Hence 
$$
p_\ad(\vz)=\det(z_0I+\sum_{i=1}^s z_i\ad_{e_i})=\det(z_0I)=z_0^n.
$$
\indent Conversely, suppose $p_\ad(\vz)=z_0^n$. By Lemma \ref{D}, there exists an invertible matrix $D$ of the form
$$
D=
\begin{pmatrix}
    1 & 0 \\
    0 & P^\top
\end{pmatrix}
$$
where $P$ is a transition matrix, such that 
$$
p_{\ad(L)}(\vz D)=p_\ad(\vz)=z_0^n.
$$
Since the matrix $D$ only rewrite $z_i$ for all $i\in\{1,2,\ldots,s\}$ but leave $z_0$ invariant, we get $p_{\ad(L)}(\vz)=z_0^n$ as well. By Proposition \ref{nilpotent1}, $\ad(L)\cong\quo{L}{Z(L)}$ is nilpotent. Hence $L$ is also nilpotent.
\end{proof}

\medskip

We can apply Theorem \ref{nilpotent} to determine if a finite-dimensional Lie algebra is nilpotent. A Lie algebra $L$ in the next example is indeed nilpotent, but the nilpotency is not obvious from bracket relations.

\medskip

\begin{ex}
    Let $L=\spann\{e_1,e_2,e_3,e_4,e_5\}$ be a Lie algebra defined by
    \begin{multicols}{2}
    \noindent
    \begin{align*}
        [e_1,e_2]&=e_1+e_2+e_3-e_4-e_5,\\
        [e_1,e_3]&=-e_2-e_3+e_4,\\
        [e_1,e_4]&=e_3,\\
        [e_1,e_5]&=e_2+e_3-e_4,\\
        [e_2,e_3]&=-e_1+e_3+e_5,\\
        [e_2,e_4]&=-e_1+e_3+e_5,\\
        [e_2,e_5]&=-e_2-2e_3+e_4,\\
        [e_3,e_4]&=e_1-e_3-e_5,\\
        [e_3,e_5]&=e_2+e_3-e_4,\\
        [e_4,e_5]&=e_1-2e_3-e_5.
    \end{align*}
    \end{multicols}
    Then we have
    \begin{multicols}{2}
    \noindent
    \begin{align*}
    \scalemath{0.85}{
        \ad_{e_1}=
        \begin{pmatrix}
            \ph0&\ph1&\ph0&\ph0&\ph0\ph\\
            \ph0&\ph1&-1&\ph0&\ph1\ph\\
            \ph0&\ph1&-1&\ph1&\ph1\ph\\
            \ph0&-1&\ph1&\ph0&-1\ph\\
            \ph0&-1&\ph0&\ph0&\ph0\ph
        \end{pmatrix},
    }\\
    \scalemath{0.85}{
        \ad_{e_2}=
        \begin{pmatrix}
            -1&\ph0&-1&-1&\ph0\ph\\
            -1&\ph0&\ph0&\ph0&-1\ph\\
            -1&\ph0&\ph1&\ph1&-2\ph\\
            \ph1&\ph0&\ph0&\ph0&\ph1\ph\\
            \ph1&\ph0&\ph1&\ph1&\ph0\ph
        \end{pmatrix},
    }\\
    \scalemath{0.85}{
        \ad_{e_3}=
        \begin{pmatrix}
            \ph0&\ph1&\ph0&\ph1&\ph0\ph\\
            \ph1&\ph0&\ph0&\ph0&\ph1\ph\\
            \ph1&-1&\ph0&-1&\ph1\ph\\
            -1&\ph0&\ph0&\ph0&-1\ph\\
            \ph0&-1&\ph0&-1&\ph0\ph
        \end{pmatrix},
    }\\
    \scalemath{0.85}{
        \ad_{e_4}=
        \begin{pmatrix}
            \ph0&\ph1&-1&\ph0&\ph1\ph\\
            \ph0&\ph0&\ph0&\ph0&\ph0\ph\\
            -1&-1&\ph1&\ph0&-2\ph\\
            \ph0&\ph0&\ph0&\ph0&\ph0\ph\\
            \ph0&-1&\ph1&\ph0&-1\ph
        \end{pmatrix},
    }\\
    \scalemath{0.85}{
        \ad_{e_5}=
        \begin{pmatrix}
            \ph0&\ph0&\ph0&-1&\ph0\ph\\
            -1&\ph1&-1&\ph0&\ph0\ph\\
            -1&\ph2&-1&\ph2&\ph0\ph\\
            \ph1&-1&\ph1&\ph0&\ph0\ph\\
            \ph0&\ph0&\ph0&\ph1&\ph0\ph
        \end{pmatrix}
    }
    \end{align*}
    \end{multicols}
    and the characteristic polynomial $p_\ad(\vz)$ is
    \begin{align*}
    \scalemath{0.7}{
    \det
    \begin{pmatrix}
            z_0-z_2&z_1+z_3+z_4&-z_2-z_4&-z_2+z_3-z_5&z_4\\
            -z_2+z_3-z_5&z_0+z_1+z_5&-z_1-z_5&0&z_1-z_2+z_3\\
            -z_2+z_3-z_4-z_5&z_1-z_3-z_4+2z_5&z_0-z_1+z_2+z_4-z_5&z_1+z_2-z_3+2z_5&z_1-2z_2+z_3-2z_4\\
            z_2-z_3+z_5&-z_1-z_5&z_1+z_5&z_0&-z_1+z_2-z_3\\
            z_2&-z_1-z_3-z_4&z_2+z_4&z_2-z_3+z_5&z_0-z_4
    \end{pmatrix}
    }
    \end{align*}
    which exactly equal to $z_0^5$. Hence $L$ is nilpotent by Theorem \ref{nilpotent}.\\
    \indent In fact, $L$ is isomorphic to a Lie algebra $\mathfrak{n}_2^5=\spann\{x_1,x_2,x_3,x_4,x_5\}$ where $[x_1,x_2]=x_3, [x_1,x_3]=x_4, [x_1,x_4]=x_5, [x_2,x_3]=x_5$ (see \cite{goze}) via an isomorphism $\varphi:\mathfrak{n}_2^5\to L$ defined by $\varphi(x_1)=e_1-e_3,\varphi(x_2)=e_2+e_3, \varphi(x_3)=e_3, \varphi(x_4)=-e_2-e_3+e_4$ and $\varphi(x_5)=-e_1+e_3+e_5$.
\end{ex}

\medskip

Lastly, using the homogeneity of the characteristic polynomial, to check the nilpotency of a Lie algebra by our criterion, we can ignore the $z_0$ variable in the calculation of the characteristic polynomial of the adjoint representation. We have the following characterization of finite-dimensional nilpotent Lie algebras. 

\medskip

\begin{cor}
    Let $L$ be a Lie algebra with a basis $\{e_1,e_2,\ldots,e_s\}$. Then $L$ is nilpotent if and only if \ $\det(I+\sum_{i=1}^s z_i \ad_{e_i})=1$.
\end{cor}

\begin{proof}
    The sufficient condition is obvious. Suppose that $L$ is not nilpotent. Then $p_\ad(\vz)=\det(z_0I+\sum_{i=1}^s z_i \ad_{e_i})\neq z_0^n$ by Theorem \ref{nilpotent}. Since $p_\ad(\vz)$ is a homogeneous polynomial of degree $n$, it can be expressed as
    $$
    p_\ad(\vz)=\sum_{i=0}^n f_iz_0^{n-i}\in\ff[\vz]
    $$
    where $f_i\in\ff[z_1,\ldots,z_s]$ is a homogeneous polynomial of degree $i$ for all $i\in\{0,1,\ldots,n\}$. Since $p_\ad(\vz)\neq z_0^n$, either $f_0\neq1$ or $f_{i_0}\neq0$ for some $i_0\in\{1,2,\ldots,n\}$. If $f_0\neq1$, then
    $$
    p_\ad(1,z_1,\ldots,z_s)=\sum_{i=0}^n f_i=f_0+\sum_{i=1}^n f_i\neq1.
    $$
    On the other hand, if there exists $i_0\in\{1,2,\ldots,n\}$ such that $f_{i_0}\neq0$, then $p_\ad(1,z_1,\ldots,z_s)=\sum_{i=0}^n f_i$ must contain a homogeneous term $f_{i_0}$ of degree $i_0\geq1$. Consequently, $p_\ad(1,z_1,\ldots,z_s)\neq1$. This completes the proof.
\end{proof}

\bibliographystyle{alpha}
\bibliography{References.bib}
\end{document}